\documentclass[10pt]{article}
\usepackage{amsthm}
\usepackage{amsmath}
\usepackage{amssymb}
\usepackage{makeidx}

\theoremstyle{definition}
\newtheorem{definition}{Definition}[section]

\newtheorem{example}[definition]{Example}

\newtheorem{remark}[definition]{Remark}

\theoremstyle{plain}
\newtheorem{theorem}[definition]{Theorem}
\newtheorem{lemma}[definition]{Lemma}
\newtheorem{proposition}[definition]{Proposition}
\newtheorem{corollary}[definition]{Corollary}

\newcommand{\lcm}{{\rm lcm}}

\numberwithin{equation}{section}

\def\N{{\mathbb N}}

\def\P{{\mathbb P}}
\def\I{{\mathbb I}}

\begin{document}
\title{Cancelling Congruences of Lattices, While\\ Keeping Their Numbers of Filters and Ideals}
\author{Claudia MURE\c SAN\\
{\small University of Cagliari}\\ 
{\small c.muresan@yahoo.com}\\ \\ 
{\em Dedicated to Professor George Georgescu}}
\date{\today }
\maketitle

\begin{abstract} In this paper, we study the congruences, prime filters and prime ideals of horizontal sums of bounded lattices, then, through a construction based on horizontal sums and without enforcing the Continuum Hypothesis, we are modifying an example from \cite{eucard} into a solution to the problem we have proposed in the same article: finding a lattice with the cardinalities of the sets of filters, ideals and congruences pairwise distinct.

\end{abstract}

\section{Introduction}
\label{introduction}

In \cite{eucard}, we have proposed the following problem: finding lattices with the cardinalities of the sets of congruences, filters and ideals pairwise distinct, or disproving their existence. In this paper, by using horizontal sums, we are modifying an example from \cite{eucard} of a lattice with the set of the filters countable and the set of the ideals uncountable into a simple lattice with the same numbers of filters and ideals. To cancel the congruences of this lattice, we are using a construction inspired by the method of constructing simple orthomodular lattices through horizontal sums from the proof of \cite[Proposition $5.11$]{bruhar}. Our method involves the use of multiple horizontal sums to turn arbitrary bounded lattices into simple bounded lattices.

We are also studying the effect of the basic horizontal sum construction on congruences and prime filters and ideals of bounded lattices, then apply it to a lattice with the set of the filters countable and the set of the ideals uncountable which can be turned into a simple lattice through a single horizontal sum. Note that, while many of the results on cardinalities from \cite{eucard} only hold under the Generalized Continuum Hypothesis, all results obtained in the present paper are valid without enforcing the Continuum Hypothesis.

\section{Preliminaries}
\label{preliminaries}

Throughout this paper, whenever there is no danger of confusion, we designate algebras by their underlying sets. $\N $ will denote the set of the natural numbers, $\N ^*=\N \setminus \{0\}$ and $\P $ will be the set of the prime natural numbers. For any sets $M$ and $N$, we denote by $M\amalg N$ the disjoint union of $M$ and $N$, by ${\cal P}(M)$ the set of the subsets of $M$ and by $|M|$ the cardinality of $M$. Also, for any cardinality $\kappa $, we denote by ${\cal P}_{\kappa }(M)=\{S\in {\cal P}(M)\ |\ |S|=\kappa \}$, ${\cal P}_{<\kappa }(M)=\{S\in {\cal P}(M)\ |\ |S|<\kappa \}$ and ${\cal P}_{\leq \kappa }(M)=\{S\in {\cal P}(M)\ |\ |S|\leq \kappa \}$; note that, if $0<\kappa \leq |M|$, then $|{\cal P}_{\kappa }(M)|=|M|^{\kappa }$, so $\displaystyle |M|^{\kappa -1}\leq |{\cal P}_{<\kappa }(M)|=\sum _{0\leq \iota <\kappa }|M|^{\iota }\leq \kappa \cdot |M|^{\kappa }$ and $|M|^{\kappa }=|{\cal P}_{\kappa }(M)|\leq |{\cal P}_{\leq \kappa }(M)|=|{\cal P}_{\kappa }(M)|+|{\cal P}_{<\kappa }(M)|\leq (\kappa +1)\cdot |M|^{\kappa }$, hence, if $|M|$ is infinite and $0<\kappa \leq |M|$, then $|{\cal P}_{\kappa }(M)|=|{\cal P}_{<\kappa }(M)|=|{\cal P}_{\leq \kappa }(M)|=|M|^{\kappa }$.

For any non--empty set $M$, $({\rm Eq}(M),\vee ,\cap ,\subseteq ,\Delta _{M},\nabla _M)$ and $({\rm Part}(M),\vee ,\wedge ,\leq ,\{\{x\}\ |\ x\in M\},\{M\})$ will be the bounded lattices of the equivalences and the partitions of $M$, respectively, and $eq:{\rm Part}(M)\rightarrow {\rm Eq}(M)$ shall be the canonical lattice isomorphism. If $n\in \N ^*$ and $\pi =\{M_1,\ldots ,M_n\}\in {\rm Part}(M)$, then the equivalence $eq(\{M_1,\ldots ,M_n\})$ will be denoted, simply, by $eq(M_1,\ldots ,M_n)$.

Let $L$ be a lattice. Then $\prec $ will denote the cover relation in $L$. For any $U\subseteq L$ and any $a,b\in L$, $[U)_L$ and $(U]_L$ shall be the filter, respectively the ideal of $L$ generated by $U$, and we use the common notations $[a)_L=[\{a\})_L$, $(a]_L=(\{a\}]_L$ and $[a,b]_L=[a)_L\cap (b]_L$. If the index $L$ is omitted, then the interval $[a,b]$ is considered in the lattice $\N $ with the natural order.

${\rm Con}(L)$, ${\rm Filt}(L)$, ${\rm PFilt}(L)$, ${\rm Id}(L)$ and ${\rm PId}(L)$ shall be the lattices of the congruences, filters, principal filters, ideals and principal ideals of $L$, respectively. Recall that the {\em prime congruences} of $L$ are the prime elements of the lattice ${\rm Con}(L)$, so all maximal congruences of $L$ are prime congruences. We denote by ${\rm Max}(L)$, ${\rm Spec}(L)$, ${\rm Spec}_{\rm Filt}(L)$ and ${\rm Spec}_{\rm Id}(L)$ the sets of the maximal congruences, prime congruences, prime filters and prime ideals of $L$, respectively. Recall that each class of a congruence of $L$ is a convex sublattice of $L$, thus it is the intersection of a filter and an ideal of $L$. If $L$ is a bounded lattice, then we denote by ${\rm Con}_{01}(L)$ the set of the congruences of $L$ whose classes of $0$ and $1$ are singletons: ${\rm Con}_{01}(L)=\{\theta \in {\rm Con}(L)\ |\ 0/\theta =\{0\},1/\theta =\{1\}\}$.

For any $n\in \N ^*$, ${\cal L}_n$ shall be the $n$--element chain. We use the common notations $M_3$ for the diamond and $N_5$ for the pentagon. For any lattices $K$ and $L$, the notation $K\cong L$ will specify the fact that $K$ and $L$ are isomorphic. We abbreviate by {\em DCC} the descending chain condition.

Recall that the {\em ordinal sum} of a lattice $(L,\leq ^L,1^L)$ with last element and a lattice $(M,\leq ^M,0^M)$ with first element is the lattice denoted by $L\oplus M$ obtained by identifying $1^L=0^M$ and glueing $L$ and $M$ at this single common point. More precisely, we let $\varepsilon =eq(\{\{1^L,0^M\}\}\cup \{\{x\}\ |\ x\in (L\setminus \{1^L\})\amalg (M\setminus \{1^M\})\})\in {\rm Eq}(L\amalg M)$ and consider the set $L\oplus M=(L\amalg M)/\varepsilon $. Since $\varepsilon \cap L^2=\Delta _L\in {\rm Con}(L)$ and $\varepsilon \cap M^2=\Delta _M\in {\rm Con}(M)$, we may identify $L$ with $L/\varepsilon \cong L$ and $M$ with $M/\varepsilon \cong M$ by identifying $x$ with $x/\varepsilon $ for all $x\in L\amalg M$. Now we define the lattice order $\leq ^{L\oplus M}=\leq ^L\cup \leq ^M$ on $L\oplus M$. Clearly, the ordinal sum of bounded lattices is associative.

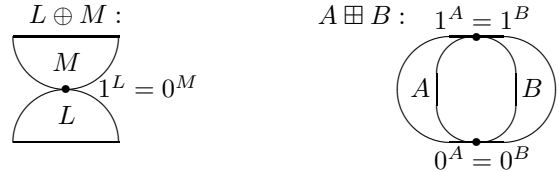
\begin{center}\begin{tabular}{cc}
\begin{picture}(40,50)(0,0)
\put(20,20){\circle*{3}}
\put(32,17){$1^L=0^M$}
\put(0,0){\line(1,0){40}}
\put(0,40){\line(1,0){40}}
\put(20,0){\oval(40,40)[t]}
\put(20,40){\oval(40,40)[b]}
\put(17,7){$L$}
\put(15,27){$M$}
\put(6,45){$L\oplus M:$}
\end{picture}
&\hspace*{100pt}
\begin{picture}(40,50)(0,0)
\put(20,0){\circle*{3}}
\put(20,40){\circle*{3}}
\put(20,20){\oval(60,40)}
\put(20,20){\oval(30,40)}
\put(-5,17){$A$}
\put(37,17){$B$}
\put(4,43){$1^A=1^B$}
\put(4,-10){$0^A=0^B$}
\put(-40,45){$A\boxplus B:$}
\end{picture}\end{tabular}\end{center}\vspace*{2pt}

Recall that the {\em horizontal sum} of two non--trivial bounded lattices $(A,\leq ^A,0^A,1^A)$ and $(B,\leq ^B,0^B,1^B)$ is the non--trivial bounded lattice denoted $A\boxplus B$ and obtained by glueing $A$ and $B$ at their first elements and at their last elements. We can generalize this construction to an arbitrary non--empty family $((A_i,\leq ^{A_i},$\linebreak $0^{A_i},1^{A_i}))_{i\in I}$ of non--trivial bounded lattices. For the precise definition, we let $\displaystyle \xi =eq(\{\{0^{A_i}\ |\ i\in I\},\{1^{A_i}\ |\ i\in I\}\}\cup \{\{x\}\ |\ x\in \amalg _{i\in I}(A_i\setminus \{0^{A_i},1^{A_i}\})\})\in {\rm Eq}(\amalg _{i\in I}A_i)$ and consider the set $\displaystyle \boxplus _{i\in I}A_i=(\amalg _{i\in I}A_i)/\xi $. Since, for every $i\in I$, $\xi \cap A_i^2=\Delta _{A_i}\in {\rm Con}(A_i)$, we may identify each $A_i$ with $A_i/\xi \cong A_i$ by identifying $x$ with $x/\xi $ for all $\displaystyle x\in \amalg _{i\in I}A_i$. Now we define the lattice order $\displaystyle \leq ^{\boxplus _{i\in I}A_i}=\bigcup _{i\in I}\leq ^{A_i}$ on $\boxplus _{i\in I}A_i$; the lattice $(\boxplus _{i\in I}A_i,\leq ^{\boxplus _{i\in I}A_i})$ has the first element $0=0^{\boxplus _{i\in I}A_i}=0^{A_j}$ and the last element $1=1^{\boxplus _{i\in I}A_i}=1^{A_j}$ for every $j\in I$. If $\alpha _i\in {\rm Eq}(A_i)\setminus \{\nabla _{A_i}\}$ for all $i\in I$, then we denote by $\displaystyle \boxplus _{i\in I}\alpha _i=eq(\bigcup _{i\in I}(A_i/\alpha _i\setminus \{0^{A_i}/\alpha _i,1^{A_i}/\alpha _i\})\cup \{\bigcup _{i\in I}0^{A_i}/\alpha _i,\bigcup _{i\in I}1^{A_i}/\alpha _i\})\in {\rm Eq}(\boxplus _{i\in I}A_i)\setminus \{\nabla _{\boxplus _{i\in I}A_i}\}$; so $\boxplus _{i\in I}\alpha _i$ is the equivalence on $\boxplus _{i\in I}A_i$ whose classes are: $\displaystyle x/(\boxplus _{i\in I}\alpha _i)=\begin{cases}x/\alpha _i, & x\in A_i\setminus \{0,1\}\mbox{ for some }i\in I,\\ \bigcup _{i\in I}x/\alpha _i, & x\in \{0,1\}.\end{cases}$ Note that ${\cal L}_2\boxplus B=B$ and $\Delta _{{\cal L}_2}\boxplus \beta =\beta $ for any non--trivial bounded lattice $B$ and any $\beta \in {\rm Con}(B)\setminus \{\nabla _B\}$. Clearly, the horizontal sum of non--trivial bounded lattices is commutative and associative, and so is the operation $\boxplus $ on proper equivalences on the underlying sets of those lattices.

\section{Some Introductory Remarks}
\label{theproblem}

Let $L$ be a lattice. Then $|L|=|{\rm PFilt}(L)|=|{\rm PId}(L)|\leq |{\rm Filt}(L)|,|{\rm Id}(L)|\leq |{\cal P}(L)|=2^{|L|}$, while $|{\rm Con}(L)|\leq |{\rm Eq}(L)|=|{\rm Part}(L)|\leq |\{\pi \in {\cal P}({\cal P}(L))\ |\ |\pi |\leq |L|\}|=|{\cal P}_{\leq |L|}({\cal P}(L))|\leq (|L|+1)\cdot (2^{|L|})^{|L|}=(|L|+1)\cdot (2^{|L|\cdot |L|})$.

If all filters of $L$ are principal, then $|{\rm Filt}(L)|=|L|$, and the same holds for ideals, but the converses of these implications do not hold, as shown by a set of examples in \cite[Remarks $5.2$ and $5.3$]{eucard}. 

If $L$ is finite, then all its filters and ideals are principal, so $|{\rm Filt}(L)|=|{\rm Id}(L)|=|L|$, while, if $L$ is infinite, then, by the above, $|L|\leq |{\rm Filt}(L)|,|{\rm Id}(L)|\leq 2^{|L|}$ and $|{\rm Con}(L)|\leq 2^{|L|}$. Therefore a lattice $L$ with $|{\rm Con}(L)|$, $|{\rm Filt}(L)|$ and $|{\rm Id}(L)|$ pairwise distinct has to be infinite and, if we enforce the Generalized Continuum Hypothesis, then we must have $|{\rm Con}(L)|<|L|$ and $\{|{\rm Filt}(L)|,|{\rm Id}(L)|\}=\{|L|,2^{|L|}\}$, so $L$ has to have strictly less congruences than elements and either as many filters as elements and as many ideals as subsets or vice--versa.

Let $F$ be a filter of $L$. Then $F$ is principal iff it has a minimum, case in which $F=[\min (F))_L$. $L\setminus F$ is an ideal of $F$ iff $F$ is prime. The duals of these hold for ideals. The map $P\mapsto L\setminus P$ is a bijection between ${\rm Spec}_{\rm Filt}(L)$ and ${\rm Spec}_{\rm Id}(L)$ and, for any prime filter $P$ of $L$, $eq(P,L\setminus P)\in {\rm Max}(L)\subseteq {\rm Spec}(L)$, hence, for any non--empty family $(P_i)_{i\in I}$ of prime filters of $L$, if we denote by $\displaystyle \theta =\bigcap _{i\in I}eq(P_i,L\setminus P_i)$, then $\theta $ is a congruence of $L$ such that $L/\theta $ is a bounded lattice, with the filter $\displaystyle \bigcap _{i\in I}P_i$ as top element and the ideal $\displaystyle \bigcap _{i\in I}(L\setminus P_i)$ as bottom element. Therefore $L$ has at least as many congruences as intersections of prime filters and at least as many congruences as intersections of prime ideals. In the particular case when $L$ is distributive, $L$ has at least as many congruences as filters and at least as many congruences as ideals, so, if $L$ is an infinite distributive lattice, then $|L|\leq |{\rm Filt}(L)|,|{\rm Id}(L)|\leq |{\rm Con}(L)|\leq 2^{|L|}$, hence the cardinalities $|{\rm Filt}(L)|$, $|{\rm Id}(L)|$ and $|{\rm Con}(L)|$ can not be pairwise distinct under the Generalized Continuum Hypothesis.

Since ${\rm Con}(L)$ is a complete sublattice of ${\rm Eq}(L)$, for any non--empty family $(\pi _i)_{i\in I}\subseteq {\rm Part}(L)$, if $eq(\pi _i)\in {\rm Con}(L)$ for all $i\in I$, then $\displaystyle eq(\bigvee _{i\in I}\pi _i)=\bigvee _{i\in I}eq(\pi _i),eq(\bigwedge _{i\in I}\pi _i)=\bigcap _{i\in I}eq(\pi _i)\in {\rm Con}(L)$. If $S$ is a non--empty subset of $L$ and $\sigma \in {\rm Part}(L)$ such that $\sigma \subseteq \pi _i$ for all $i\in I$, then $\displaystyle \sigma \subseteq \bigvee _{i\in I}\pi _i$ and $\displaystyle \sigma \subseteq \bigwedge _{i\in I}\pi _i$; also, if $\nu ,\rho, \pi \in {\rm Part}(L)$ are such that $\nu \leq \rho \leq \pi $, $\sigma \subseteq \nu $ and $\sigma \subseteq \pi $, then $\sigma \subseteq \rho $. Therefore $\{\theta \in {\rm Con}(L)\ |\ \sigma \subseteq L/\theta \}$ is a complete convex sublattice of ${\rm Con}(L)$, so it is a bounded lattice. In particular, if $L$ is a bounded lattice, then ${\rm Con}_{01}(L)$ is a complete convex sublattice of ${\rm Con}(L)$ which obviously contains $\Delta _L$, hence ${\rm Con}_{01}(L)$ is a principal ideal of ${\rm Con}(L)$, generated by the largest congruence $\mu $ of $L$ with the classes of $0$ and $1$ singletons.

\section{Horizontal Sums Cancel Congruences, Prime Filters and Prime Ideals, While Leaving Filters and Ideals in Place}
\label{horizsums}

Throughout this section, $L$ shall be a non--trivial bounded lattice.

We will sometimes use the remarks in this paper without referencing them.

\begin{remark} For any proper filter $P$ of $L$, the following are equivalent:\begin{enumerate}
\item\label{prime1} $P$ is a prime filter of $L$;
\item\label{prime2} $L\setminus P$ is an ideal of $L$;
\item\label{prime3} $L\setminus P$ is a prime ideal of $L$;
\item\label{prime4} $eq(P,L\setminus P)$ is a congruence of $L$;
\item\label{prime5} $eq(P,L\setminus P)$ is a maximal congruence of $L$.\end{enumerate}

Indeed, (\ref{prime4}) and (\ref{prime5}) are clearly equivalent, and so are (\ref{prime1}), (\ref{prime2}) and (\ref{prime3}). It is straightforward that (\ref{prime1}) and (\ref{prime3}) imply (\ref{prime4}). If $eq(P,L\setminus P)\in {\rm Con}(L)$, then $L\setminus P=0/eq(P,L\setminus P)\in {\rm Id}(L)$, so (\ref{prime4}) implies (\ref{prime5}).

Note, also, from the above, that, for any congruence $\theta $ of $L$, since $0/\theta \in {\rm Filt}(L)$ and $1/\theta \in {\rm Id}(L)$, we have: $|L/\theta |=2$ iff $\theta =eq(0/\theta ,1/\theta )\neq \nabla _L$ iff $\theta \neq \nabla _L$ and $0/\theta \cup 1/\theta =L$, which implies that $0/\theta \in {\rm Spec}_{\rm Filt}(L)$ and $1/\theta \in {\rm Spec}_{\rm Id}(L)$.\label{prime}\end{remark}

\begin{lemma}\begin{enumerate}
\item\label{cgvsirred0} $0$ is meet--irreducible in $L$ iff $L\setminus \{0\}\in {\rm Filt}(L)$ iff $L\setminus \{0\}\in {\rm Spec}_{\rm Filt}(L)$ iff $\{0\}\in {\rm Spec}_{\rm Id}(L)$ iff $eq(\{0\},L\setminus \{0\})\in {\rm Con}(L)$ iff $eq(\{0\},L\setminus \{0\})\in {\rm Max}(L)$, and dually for $1$.
\item\label{cgvsirred2} If $|L|>2$, then: $0$ is meet--irreducible and $1$ is join--irreducible in $L$ iff $L\setminus \{0,1\}$ is a convex sublattice of $L$ iff $eq(\{0\},L\setminus \{0,1\},\{1\})\in {\rm Con}(L)$.\end{enumerate}\label{cgvsirred}\end{lemma}

\begin{proof} (\ref{cgvsirred0}) $L\setminus \{0\}$ is closed w.r.t. upper bounds and, for all $x,y\in L$, if $x\vee y\in L\setminus \{0\}$, then $x\in L\setminus \{0\}$ or $y\in L\setminus \{0\}$. Clearly, $L\setminus \{0\}$ is closed w.r.t. meets iff $0$ is meet--irreducible in $L$. Hence the first two equivalences hold. The rest of the equivalences follow from Remark \ref{prime}.

\noindent (\ref{cgvsirred2}) By (\ref{cgvsirred0}), if $0$ is meet--irreducible and $1$ is join--irreducible in $L$, then $eq(\{0\},L\setminus \{0\}),eq(L\setminus \{1\},\{1\})\in {\rm Con}(L)$, thus $eq(\{0\},L\setminus \{0,1\},\{1\})=eq(\{\{0\},L\setminus \{0\}\}\wedge \{L\setminus \{1\},\{1\}\})=eq(\{0\},L\setminus \{0\})\cap eq(L\setminus \{1\},\{1\})\in {\rm Con}(L)$, which in turn implies that $L\setminus \{0,1\}$ is a convex sublattice of $L$. On the other hand, if $L\setminus \{0,1\}$ is a sublattice of $L$, then it is closed w.r.t. meets, so $0$ is meet--irreducible in $L$, and w.r.t. joins, so $1$ is join--irreducible in $L$.\end{proof}

Note that, if $|L|>2$, then $eq(\{0\},L\setminus \{0,1\},\{1\})$ is not a prime congruence of $L$, because, according to Lemma \ref{cgvsirred}, $eq(\{0\},L\setminus \{0,1\},\{1\})=eq(\{0\},L\setminus \{0\})\cap eq(L\setminus \{1\},\{1\})$ is a congruence of $L$ exactly when $eq(\{0\},L\setminus \{0\})\supsetneq eq(\{0\},L\setminus \{0,1\},\{1\})$ and $eq(L\setminus \{1\},\{1\})\supsetneq eq(\{0\},L\setminus \{0,1\},\{1\})$ are congruences of $L$. Let us also notice that, if $|L|>2$, then each member of ${\rm Con}_{01}(L)$ has at least three distinct classes.

Throughout the rest of this section, $A$ and $B$ shall be non--trivial bounded lattices.

\begin{remark} If there exist $a\in A\setminus \{0,1\}$ and $b\in B\setminus \{0,1\}$, then $[\{a,b\})=A\boxplus B=(\{a,b\}]$, hence, regardless of the cardinalities of $A$ and $B$:\begin{itemize}
\item ${\rm Filt}(A\boxplus B)=({\rm Filt}(A)\setminus \{A\})\cup ({\rm Filt}(B)\setminus \{B\})\cup \{A\boxplus B\}=(({\rm Filt}(A)\cup {\rm Filt}(B))\setminus \{A,B\})\cup \{A\boxplus B\}$, and similarly for ideals, therefore, since ${\rm Filt}(A)\cap {\rm Filt}(B)=\{\{1\}\}$ and dually for ideals, we 
have:
\item $|{\rm Filt}(A\boxplus B)|=|{\rm Filt}(A)|+|{\rm Filt}(B)|-2$ and $|{\rm Id}(A\boxplus B)|=|{\rm Id}(A)|+|{\rm Id}(B)|-2$.\end{itemize}\label{filtidhsum}\end{remark}

\begin{proposition} If $|A|>2$ and $|B|>2$, then ${\rm Spec}_{\rm Filt}(A\boxplus B)\subseteq \{A\setminus \{0\},B\setminus \{0\}\}$, ${\rm Spec}_{\rm Id}(A\boxplus B)\subseteq \{A\setminus \{1\},B\setminus \{1\}\}$ and the following are equivalent:\begin{itemize}
\item $0$ is meet--irreducible in $A$ and $1$ is join--irreducible in $B$;
\item $A\setminus \{0\}\in {\rm Spec}_{\rm Filt}(A\boxplus B)$;
\item $B\setminus \{1\}\in {\rm Spec}_{\rm Id}(A\boxplus B)$;
\item $eq(A\setminus \{0\},B\setminus \{1\})\in {\rm Con}(A\boxplus B)$;
\item $eq(A\setminus \{0\},B\setminus \{1\})\in {\rm Max}(A\boxplus B)$.\end{itemize}\label{spechsum}\end{proposition}

\begin{proof} Let $P\in {\rm Filt}(A\boxplus B)\setminus \{A\boxplus B\}=({\rm Filt}(A)\setminus \{A\})\cup ({\rm Filt}(B)\setminus \{B\})$ by Remark \ref{filtidhsum}. Assume, for instance, that $P\in {\rm Filt}(A)\setminus \{A\}$. Then $P\in {\rm Spec}_{\rm Filt}(A\boxplus B)$ iff all the following hold:\begin{itemize}
\item $P\in {\rm Spec}_{\rm Filt}(A)$;
\item for all $x,y\in B$, $x\vee y\in P\cap B=\{1\}$ implies $x\in P\cap B=\{1\}$ or $y\in P\cap B=\{1\}$, which is equivalent to $1$ being join--irreducible in $B$, which in turn is equivalent to $B\setminus \{1\}\in {\rm Spec}_{\rm Id}(B)$ by Lemma \ref{cgvsirred}, (\ref{cgvsirred0});
\item for all $a\in A\setminus \{0,1\}$ and all $b\in B\setminus \{0,1\}$, if $a\vee b\in P$, then $a\in P$ or $b\in P$, so that $a\in P$ since $b\in B\setminus A$, which is equivalent to $A\setminus \{0,1\}\subseteq P$ and thus to $A\setminus \{0\}\subseteq P$ since $P$ is a filter of $A$, which in turn is equivalent to $P=A\setminus \{0\}$ since $P$ is a proper filter.\end{itemize}

Therefore $P\in {\rm Spec}_{\rm Filt}(A\boxplus B)$ iff $1$ is join--irreducible in $B$ and $P=A\setminus \{0\}\in {\rm Spec}_{\rm Filt}(A)$ iff $P=A\setminus \{0\}$ and $0$ is meet--irreducible in $A$ and $1$ is join--irreducible in $B$, again by Lemma \ref{cgvsirred}, (\ref{cgvsirred0}). Dually, a proper ideal $Q$ of $A\boxplus B$ is prime iff $Q=B\setminus \{1\}$ and $0$ is meet--irreducible in $A$ and $1$ is join--irreducible in $B$. From the above, the fact that $B\setminus \{1\}=(A\boxplus B)\setminus (A\setminus \{0\})$ and Remark \ref{prime}, we obtain the equivalences in the enunciation. 

Similarly, if $P\in {\rm Filt}(B)\setminus \{B\}$, then $P=B\setminus \{0\}$, hence ${\rm Spec}_{\rm Filt}(A\boxplus B)\subseteq \{A\setminus \{0\},B\setminus \{0\}\}$. Dually, ${\rm Spec}_{\rm Id}(A\boxplus B)\subseteq \{A\setminus \{1\},B\setminus \{1\}\}$.\end{proof}

\begin{remark} For any $\theta \in {\rm Con}(A\boxplus B)$, we have: $\theta \cap A^2\in {\rm Con}(A)$, $\theta \cap B^2\in {\rm Con}(B)$ and: $\theta =\nabla _{A\boxplus B}$ iff $(0,1)\in \theta $ iff $(0,1)\in \theta \cap A^2$ iff $(0,1)\in \theta \cap B^2$ iff $\theta \cap A^2=\nabla _A$ iff $\theta \cap B^2=\nabla _B$, and, if $\theta \neq \nabla _{A\boxplus B}$, then $\theta =(\theta \cap A^2)\boxplus (\theta \cap B^2)$.\label{cg01}\end{remark}

\begin{lemma} ${\rm Con}_{01}(A\boxplus B)=\{\alpha \boxplus \beta \ |\ \alpha \in {\rm Con}_{01}(A),\beta \in {\rm Con}_{01}(B)\}\cong {\rm Con}_{01}(A)\times {\rm Con}_{01}(B)$.\label{cg01hsum}\end{lemma}

\begin{proof} By Remark \ref{cg01}, the fact that ${\rm Con}_{01}(A)\subseteq {\rm Con}(A)\setminus \{\nabla _A\}$ and the same for $B$ and $A\boxplus B$, and the definition of the horizontal sum of proper congruences, according to which $0/(\alpha \boxplus \beta )=0/\alpha \cup 0/\beta $ and $1/(\alpha \boxplus \beta )=1/\alpha \cup 1/\beta $ for all $\alpha \in {\rm Con}(A)\setminus \{\nabla _A\}$ and all $\beta \in {\rm Con}(B)\setminus \{\nabla _B\}$, we get that ${\rm Con}_{01}(A\boxplus B)=\{\alpha \boxplus \beta \ |\ \alpha \in {\rm Con}_{01}(A),\beta \in {\rm Con}_{01}(B)\}$, hence, the map $(\alpha ,\beta )\mapsto \alpha \boxplus \beta $ from ${\rm Con}_{01}(A)\times {\rm Con}_{01}(B)$ to ${\rm Con}_{01}(A\boxplus B)$ is surjective. By Remark \ref{cg01}, $(\alpha \boxplus \beta )\cap A^2=\alpha $ and $(\alpha \boxplus \beta )\cap B^2=\beta $ for all $\alpha \in {\rm Con}(A)\setminus \{\nabla _A\}$ and all $\beta \in {\rm Con}(B)\setminus \{\nabla _B\}$, so this map is also injective, and it is clearly order--preserving, therefore it is a lattice isomorphism.\end{proof}

\begin{theorem} If $|A|>2$ and $|B|>2$, then:\begin{enumerate}
\item\label{cghsum1} ${\rm Con}(A\boxplus B)={\rm Con}_{01}(A\boxplus B)\cup \{\nabla _{A\boxplus B}\}\cong ({\rm Con}_{01}(A)\times {\rm Con}_{01}(B))\oplus {\cal L}_2$ iff $A\boxplus B$ has no two--class congruences iff ${\rm Spec}_{\rm Filt}(A\boxplus B)=\emptyset $ iff ${\rm Spec}_{\rm Id}(A\boxplus B)=\emptyset $ iff the following conditions are fulfilled:\begin{itemize}
\item $0$ is meet--reducible in $A$ or $1$ is join--reducible in $B$, and
\item $0$ is meet--reducible in $B$ or $1$ is join--reducible in $A$;\end{itemize}

\item\label{cghsum2} ${\rm Con}(A\boxplus B)={\rm Con}_{01}(A\boxplus B)\cup \{eq(A\setminus \{0\},B\setminus \{1\}),\nabla _{A\boxplus B}\}\cong ({\rm Con}_{01}(A)\times {\rm Con}_{01}(B))\oplus {\cal L}_3$ iff $eq(A\setminus \{0\},B\setminus \{1\})$ is the unique two--class congruence of $A\boxplus B$ iff ${\rm Spec}_{\rm Filt}(A\boxplus B)=\{A\setminus \{0\}\}$ iff ${\rm Spec}_{\rm Id}(A\boxplus B)=\{B\setminus \{1\}\}$ iff the following conditions are fulfilled:\begin{itemize}
\item $0$ is meet--irreducible in $A$ and $1$ is join--irreducible in $B$, but
\item $0$ is meet--reducible in $B$ or $1$ is join--reducible in $A$;\end{itemize}

and dually for the case when $eq(A\setminus \{1\},B\setminus \{0\})$ is the unique two--class congruence of $A\boxplus B$;

\item\label{cghsum3} ${\rm Con}(A\boxplus B)={\rm Con}_{01}(A\boxplus B)\cup \{eq(A\setminus \{0\},B\setminus \{1\}),eq(A\setminus \{1\},B\setminus \{0\}),\nabla _{A\boxplus B}\}\cong ({\rm Con}_{01}(A)\times {\rm Con}_{01}(B))\oplus {\cal L}_2^2$ iff $A\boxplus B$ has two two--class congruences iff ${\rm Spec}_{\rm Filt}(A\boxplus B)=\{A\setminus \{0\},B\setminus \{0\}\}$ iff ${\rm Spec}_{\rm Id}(A\boxplus B)=\{A\setminus \{1\},B\setminus \{1\}\}$ iff $0$ is meet--irreducible in $A$ and $B$ and $1$ is join--irreducible in $A$ and $B$.\end{enumerate}\label{cghsum}\end{theorem}

\begin{proof} Of course, ${\rm Con}_{01}(A\boxplus B)\cup \{\nabla _{A\boxplus B}\}\subseteq {\rm Con}(A\boxplus B)$. Now let $\theta \in {\rm Con}(A\boxplus B)\setminus ({\rm Con}_{01}(A\boxplus B)\cup \{\nabla _{A\boxplus B}\})$, $\alpha =\theta \cap A^2\in {\rm Con}(A)\setminus \{\nabla _A\}$ and $\beta =\theta \cap B^2\in {\rm Con}(B)\setminus \{\nabla _B\}$. Then $\theta =\alpha \boxplus \beta $, so that $0/\theta =0/\alpha \cup 0/\beta $ and $1/\theta =1/\alpha \cup 1/\beta $. By the choice of $\theta $, we have $0/\theta \supsetneq \{0\}$ or $1/\theta \supsetneq \{1\}$, hence $0/\alpha \supsetneq \{0\}$ or $0/\beta \supsetneq \{0\}$ or $1/\alpha \supsetneq \{1\}$ or $1/\beta \supsetneq \{1\}$.

Assume, for instance, that $0/\alpha \supsetneq \{0\}$, so that there exists an $a\in A\setminus \{0\}$ with $(0,a)\in \alpha \subseteq \theta $. Since $\theta \neq \nabla _{A\boxplus B}$, we have $a\neq 1$. Let $b\in B\setminus \{0\}$ and $c\in A\setminus \{1\}$, arbitrary. Then $(b,1)=(0\vee b,a\vee b)\in \theta $, thus $(0,c)=(c\wedge b,c\wedge 1)\in \theta $. Hence $B\setminus \{0\}\subseteq 1/\theta $ and $A\setminus \{1\}\subseteq 0/\theta $, therefore $\nabla _{A\boxplus B}\supseteq \theta \supseteq eq(A\setminus \{1\},B\setminus \{0\})\in {\rm Max}(A\boxplus B)$ (see Remark \ref{prime}), hence $\theta =eq(A\setminus \{1\},B\setminus \{0\})$. Dually, if $1/\beta \supsetneq \{1\}$, then we also get $\theta =eq(A\setminus \{1\},B\setminus \{0\})$.

Similarly, $0/\beta \supsetneq \{0\}$ iff $1/\alpha \supsetneq \{1\}$ iff $\theta =eq(A\setminus \{0\},B\setminus \{1\})$.

Therefore ${\rm Con}(A\boxplus B)\subseteq {\rm Con}_{01}(A\boxplus B)\cup \{eq(A\setminus \{0\},B\setminus \{1\}),eq(A\setminus \{1\},B\setminus \{0\}),\nabla _{A\boxplus B}\}$, and we get the four cases in the enunciation, with the form of the prime spectra of filters and ideals of $L$ following from Remark \ref{prime}, the conditions on the meet--irreducibility of $0$ and the join--irreducibility of $1$ being inferred from Proposition \ref{spechsum} and the shape of the lattice ${\rm Con}(A\boxplus B)$ being given by Lemma \ref{cg01hsum} and the fact that ${\rm Con}_{01}(A\boxplus B)$ is a bounded lattice.\end{proof}

\begin{corollary} If $|A|>2$ and $|B|>2$, then $A\boxplus B$ is subdirectly irreducible iff one of the following conditions is satisfied:\begin{itemize}
\item ${\rm Con}_{01}(A)=\{\Delta _A\}$, ${\rm Con}_{01}(B)=\{\Delta _B\}$ and $A\boxplus B$ has at most one two--class congruence;
\item ${\rm Con}_{01}(A)=\{\Delta _A\}$ and ${\rm Con}_{01}(B)$ has a single atom, or vice--versa.\end{itemize}\end{corollary}

\begin{example} The following horizontal sums have the prime spectra of filters and ideals and the congruence lattices below:\vspace*{-17pt}

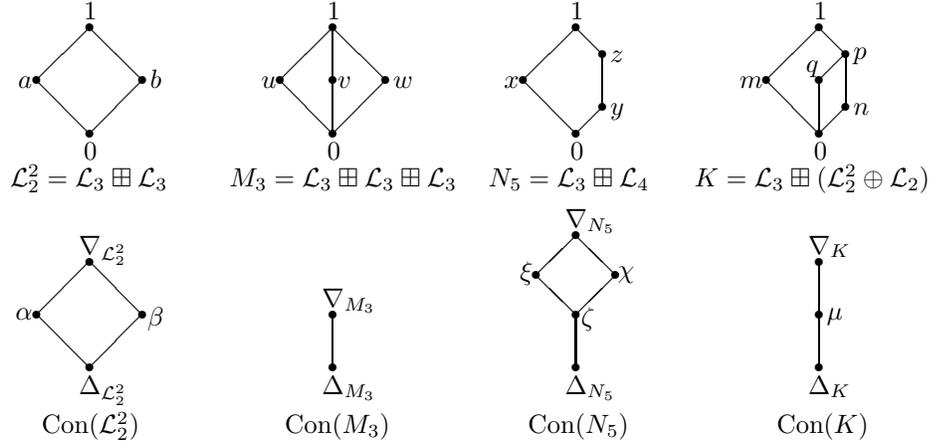
\begin{center}
\begin{tabular}{cccc}
\begin{picture}(80,80)(0,0)
\put(38,10){$0$}
\put(40,20){\circle*{3}}
\put(20,40){\circle*{3}}
\put(60,40){\circle*{3}}
\put(40,60){\circle*{3}}
\put(40,20){\line(-1,1){20}}
\put(40,20){\line(1,1){20}}
\put(40,60){\line(-1,-1){20}}
\put(40,60){\line(1,-1){20}}
\put(13,37){$a$}
\put(63,37){$b$}
\put(38,63){$1$}
\put(10,0){${\cal L}_2^2={\cal L}_3\boxplus {\cal L}_3$}
\end{picture}
& 
\begin{picture}(80,80)(0,0)
\put(38,10){$0$}
\put(40,20){\circle*{3}}
\put(20,40){\circle*{3}}
\put(40,40){\circle*{3}}
\put(60,40){\circle*{3}}
\put(40,60){\circle*{3}}
\put(40,20){\line(-1,1){20}}
\put(40,20){\line(1,1){20}}
\put(40,20){\line(0,1){40}}
\put(40,60){\line(-1,-1){20}}
\put(40,60){\line(1,-1){20}}
\put(13,37){$u$}
\put(42,37){$v$}
\put(63,37){$w$}
\put(38,63){$1$}
\put(1,0){$M_3={\cal L}_3\boxplus {\cal L}_3\boxplus {\cal L}_3$}
\end{picture}
&
\begin{picture}(80,80)(0,0)
\put(38,10){$0$}
\put(40,20){\circle*{3}}
\put(20,40){\circle*{3}}
\put(50,30){\circle*{3}}
\put(50,50){\circle*{3}}
\put(40,60){\circle*{3}}
\put(40,20){\line(-1,1){20}}
\put(40,20){\line(1,1){10}}
\put(50,30){\line(0,1){20}}
\put(40,60){\line(-1,-1){20}}
\put(40,60){\line(1,-1){10}}
\put(13,37){$x$}
\put(53,27){$y$}
\put(53,47){$z$}
\put(38,63){$1$}
\put(7,0){$N_5={\cal L}_3\boxplus {\cal L}_4$}
\end{picture}
&
\begin{picture}(80,80)(0,0)
\put(38,10){$0$}
\put(40,20){\circle*{3}}
\put(20,40){\circle*{3}}
\put(40,40){\circle*{3}}
\put(50,30){\circle*{3}}
\put(50,50){\circle*{3}}
\put(40,60){\circle*{3}}
\put(40,20){\line(-1,1){20}}
\put(40,20){\line(1,1){10}}
\put(50,30){\line(0,1){20}}
\put(40,60){\line(-1,-1){20}}
\put(40,60){\line(1,-1){10}}
\put(40,20){\line(0,1){20}}
\put(40,40){\line(1,1){10}}
\put(10,37){$m$}
\put(53,27){$n$}
\put(53,48){$p$}
\put(35,43){$q$}
\put(38,63){$1$}
\put(-7,0){$K={\cal L}_3\boxplus ({\cal L}_2^2\oplus {\cal L}_2)$}\end{picture}\\ 
\begin{picture}(80,85)(0,0)
\put(36,10){$\Delta _{{\cal L}_2^2}$}
\put(40,20){\circle*{3}}
\put(20,40){\circle*{3}}
\put(60,40){\circle*{3}}
\put(40,60){\circle*{3}}
\put(40,20){\line(-1,1){20}}
\put(40,20){\line(1,1){20}}
\put(40,60){\line(-1,-1){20}}
\put(40,60){\line(1,-1){20}}
\put(12,38){$\alpha $}
\put(62,36){$\beta $}
\put(36,63){$\nabla _{{\cal L}_2^2}$}
\put(22,-5){${\rm Con}({\cal L}_2^2)$}
\end{picture}
& 
\begin{picture}(80,85)(0,0)
\put(36,10){$\Delta _{M_3}$}
\put(40,20){\circle*{3}}
\put(40,40){\circle*{3}}
\put(40,20){\line(0,1){20}}
\put(36,43){$\nabla _{M_3}$}
\put(22,-5){${\rm Con}(M_3)$}\end{picture}
& 
\begin{picture}(80,85)(0,0)
\put(36,10){$\Delta _{N_5}$}
\put(40,20){\circle*{3}}
\put(40,40){\circle*{3}}
\put(25,55){\circle*{3}}
\put(55,55){\circle*{3}}
\put(40,70){\circle*{3}}
\put(40,20){\line(0,1){20}}
\put(40,40){\line(1,1){15}}
\put(40,40){\line(-1,1){15}}
\put(40,70){\line(1,-1){15}}
\put(40,70){\line(-1,-1){15}}
\put(19,53){$\xi $}
\put(56,54){$\chi $}
\put(42,36){$\zeta $}
\put(36,73){$\nabla _{N_5}$}
\put(22,-5){${\rm Con}(N_5)$}
\end{picture}
&
\begin{picture}(80,85)(0,0)
\put(36,10){$\Delta _K$}
\put(40,20){\circle*{3}}
\put(40,40){\circle*{3}}
\put(40,60){\circle*{3}}
\put(40,20){\line(0,1){40}}
\put(43,38){$\mu $}
\put(36,63){$\nabla _K$}
\put(24,-5){${\rm Con}(K)$}
\end{picture}\end{tabular}\end{center}

${\rm Spec}_{\rm Filt}({\cal L}_2^2)=\{\{a,1\},\{b,1\}\}$, ${\rm Spec}_{\rm Id}({\cal L}_2^2)=\{\{0,a\},\{0,b\}\}$, $\alpha =eq(\{0,a\},\{b,1\})$ and $\beta =eq(\{0,b\},\{a,1\})$.

${\rm Spec}_{\rm Filt}(M_3)=\emptyset ={\rm Spec}_{\rm Id}(M_3)$.

${\rm Spec}_{\rm Filt}(N_5)=\{\{x,1\},\{y,z,1\}\}$, ${\rm Spec}_{\rm Id}(N_5)=\{\{0,x\},\{0,y,z\}\}$, $\xi =eq(\{0,x\},\{y,z,1\})$, $\chi =eq(\{0,y,z\},$\linebreak $\{x,1\})$ and $\zeta =eq(\{0\},\{x\},\{y,z\},\{1\})$.

${\rm Spec}_{\rm Filt}(K)=\{\{m,1\}\}$, ${\rm Spec}_{\rm Id}(K)=\{\{0,n,p,q\}\}$ and $\mu =eq(\{m,1\},\{0,n,p,q\})$.

$M_3$, $N_5$ and $K$ are subdirectly irreducible.\label{exspecfiltid}\end{example}

\begin{corollary} Let $(A_i)_{i\in I}$ be a non--empty family of non--trivial bounded lattices and $H=\boxplus _{i\in I}A_i$. Then:\begin{enumerate}
\item\label{3atleast1} $\displaystyle {\rm Con}_{01}(H)=\{\boxplus _{i\in I}\alpha _i\ |\ (\forall \, i\in I)\, (\alpha _i\in {\rm Con}_{01}(A_i))\}\cong \prod _{i\in I}{\rm Con}_{01}(A_i)$;
\item\label{3atleast2} if there exist at least three distinct elements $i,j,k\in I$ with $|A_i|,|A_j|,|A_k|>2$, then ${\rm Spec}_{\rm Filt}(H)={\rm Spec}_{\rm Id}(H)=\emptyset $, $H$ has no two--class congruences, $\displaystyle {\rm Con}(H)={\rm Con}_{01}(H)\cup \{\nabla _H\}\cong (\prod _{i\in I}{\rm Con}_{01}(A_i))\oplus {\cal L}_2$ and we have the following equivalence: $H$ is subdirectly irreducible iff, for some $u\in I$, ${\rm Con}_{01}(A_t)$ has no atoms for all $t\in I\setminus \{u\}$ and ${\rm Con}_{01}(A_u)$ has at most one atom.\end{enumerate}\label{3atleast}\end{corollary}

\begin{proof} (\ref{3atleast1}) By an analogous argument to that of Lemma \ref{cg01hsum}.

\noindent (\ref{3atleast2}) By (\ref{3atleast1}), Theorem \ref{cghsum}.(\ref{cghsum1}) and the fact that $0$ is meet--reducible and $1$ is join--reducible in $A_i\boxplus A_j$ and $|\boxplus _{t\in I\setminus \{i,j\}}A_t|>2$.\end{proof}

\section{Using Multiple Horizontal Sums to Cancel All But the First and the Last Congruence, and the Application}
\label{constrthelat}

For the following to hold, we do not need to enforce the Continuum Hypothesis. Let us see a stronger construction than the horizontal sum of a bounded lattice $L$ with another bounded lattice, construction that always turns $L$ into a simple bounded lattice.

\begin{remark} Let $(L,\leq )$ be a lattice and $[a,b]_L$ be an interval of $L$ with $|[a,b]_L|>2$, which means that $a,b\in L$ are such that $a<b$ and $a\nprec b$, and let $(M_{a,b},\leq _{a,b},0_{a,b},1_{a,b})$ be a bounded lattice with $|M_{a,b}|>2$. Denote by $N$ the lattice obtained from $L$ by replacing $[a,b]_L$ with $[a,b]_L\boxplus M_{a,b}$, that is: $N=(L\amalg (M_{a,b}\setminus \{0_{a,b},1_{a,b}\}),\leq \cup \leq _{a,b}\cup \{(x,u),(u,y)\ |\ u\in M_{a,b},x\in (a]_L,y\in [b)_L\})$.

Since $[a,b]_N=[a,b]_L\boxplus M_{a,b}$ is a sublattice of $N$, for any $\theta \in {\rm Con}(N)$, we have $\theta \cap ([a,b]_N)^2\in {\rm Con}([a,b]_N)={\rm Con}([a,b]_L\boxplus M_{a,b})$, which fulfills the properties in Section \ref{horizsums}.

${\rm Filt}(N)=\{F\in {\rm Filt}(L)\ |\ a\notin F\}\cup \{F\amalg (M_{a,b}\setminus \{0_{a,b},1_{a,b}\})\ |\ F\in {\rm Filt}(L),a\in F\}\cup \{[b)_L\amalg (G\setminus \{0_{a,b},1_{a,b}\})\ |\ G\in {\rm Filt}(M_{a,b})\}$, hence $|{\rm Filt}(N)|=|{\rm Filt}(L)|+|{\rm Filt}(M_{a,b})\setminus \{\{1_{a,b}\},M_{a,b}\}|=|{\rm Filt}(L)|+|{\rm Filt}(M_{a,b})|-2$.

Similarly, ${\rm Id}(N)=\{I\in {\rm Id}(L)\ |\ b\notin I\}\cup \{I\,\amalg (M_{a,b}\setminus \{0_{a,b},1_{a,b}\})\ |\ I\in {\rm Id}(L),b\in I\}\cup \{(a]_L\amalg (J\setminus \{0_{a,b},1_{a,b}\})\ |\ J\in {\rm Id}(M_{a,b})\}$, hence $|{\rm Id}(N)|=|{\rm Id}(L)|+|{\rm Id}(M_{a,b})\setminus \{\{0_{a,b}\},M_{a,b}\}|=|{\rm Id}(L)|+|{\rm Id}(M_{a,b})|-2$.
\label{hsuminterval}\end{remark}

Throughout the rest of this section, $(L,\leq ,0,1)$ shall be a non--trivial bounded lattice. Let us apply the construction in Remark \ref{hsuminterval} to all intervals of $L$ having cardinality at least $3$, with $M_{a,b}$ replaced with ${\cal L}_2^2$. So let us denote by $D(L)$ the bounded lattice obtained from $L$ in the following way: replace each interval $I$ of $L$ with $|I|>2$ by $I\boxplus {\cal L}_2^2$. In detail, the construction of $D(L)$ can be written like this: consider a bijection from the set $\{[a,b]_L\ |\ a,b\in L,a<b,a\nprec b\}$ of the intervals of $L$ having at least three elements to a set $M$ of pairwise disjoint two--element sets, which associates to each such interval $[a,b]_L$ a two--element set $\{l_{a,b},r_{a,b}\}\in L$. Let $D(L)=(L\amalg M=L\amalg \{l_{a,b},r_{a,b}\ |\ a,b\in L,a<b,a\nprec b\},\leq \cup \Delta _M\cup \{(x,l_{a,b}),(l_{a,b},y),(x,r_{a,b}),(r_{a,b},y)\ |\ a,b\in L,a<b,a\nprec b,x\in (a]_L\cup \{l_{u,v},r_{u,v}\ |\ u,v\in L,u<v\leq a,u\nprec v\},y\in [b)_L\cup \{c_{u,v}\ |\ u,v\in L,b\leq u<v,u\nprec v\}\},0,1)$. We shall denote the order of $D(L)$ by $\leq $, as well.

\begin{example} Here is the construction above applied to the lattices ${\cal L}_3=\{0,m,1\}$, ${\cal L}_4=\{0,a,b,1\}$ and ${\cal L}_2^2\oplus {\cal L}_2=\{0,x,y,z,1\}$:\vspace*{5pt}

\begin{center}
\begin{tabular}{ccc}
\begin{picture}(80,100)(0,0)
\put(38,15){$0$}
\put(40,25){\circle*{3}}
\put(40,45){\circle*{3}}
\put(60,45){\circle*{3}}
\put(20,45){\circle*{3}}
\put(60,45){\line(-1,1){20}}
\put(60,45){\line(-1,-1){20}}
\put(20,45){\line(1,1){20}}
\put(20,45){\line(1,-1){20}}
\put(40,65){\circle*{3}}
\put(43,43){$m$}
\put(5,42){$l_{0,1}$}
\put(63,42){$r_{0,1}$}
\put(40,25){\line(0,1){40}}
\put(38,68){$1$}
\put(13,3){$D({\cal L}_3)\cong M_3$}
\end{picture}
&\hspace*{20pt}
\begin{picture}(80,100)(0,0)
\put(38,15){$0$}
\put(40,25){\circle*{3}}
\put(40,45){\circle*{3}}
\put(60,45){\circle*{3}}
\put(40,85){\circle*{3}}
\put(60,45){\line(-1,1){20}}
\put(60,45){\line(-1,-1){20}}
\put(20,45){\line(1,1){20}}
\put(20,45){\line(1,-1){20}}
\put(40,65){\circle*{3}}
\put(20,45){\circle*{3}}
\put(20,65){\circle*{3}}
\put(60,65){\circle*{3}}
\put(20,65){\line(1,1){20}}
\put(20,65){\line(1,-1){20}}
\put(60,65){\line(-1,1){20}}
\put(60,65){\line(-1,-1){20}}
\put(0,55){\circle*{3}}
\put(0,55){\line(4,3){40}}
\put(0,55){\line(4,-3){40}}
\put(80,55){\line(-4,3){40}}
\put(80,55){\line(-4,-3){40}}
\put(80,55){\circle*{3}}
\put(43,43){$a$}
\put(33,63){$b$}
\put(-15,53){$l_{0,1}$}
\put(83,52){$r_{0,1}$}
\put(12,56){$l_{a,1}$}
\put(56,50){$r_{0,b}$}
\put(25,42){$l_{0,b}$}
\put(43,64){$r_{a,1}$}
\put(40,25){\line(0,1){60}}
\put(38,88){$1$}
\put(26,3){$D({\cal L}_4)$}
\end{picture}
&\hspace*{70pt}
\begin{picture}(80,100)(0,0)
\put(38,15){$0$}
\put(40,25){\circle*{3}}
\put(40,115){\circle*{3}}
\put(-35,70){\circle*{3}}
\put(115,70){\circle*{3}}
\put(-50,68){$l_{0,1}$}
\put(118,68){$r_{0,1}$}
\put(-35,70){\line(5,3){75}}
\put(-35,70){\line(5,-3){75}}
\put(115,70){\line(-5,3){75}}
\put(115,70){\line(-5,-3){75}}
\put(10,55){\circle*{3}}
\put(70,55){\circle*{3}}
\put(10,55){\line(0,1){30}}
\put(70,55){\line(0,1){30}}
\put(10,85){\circle*{3}}
\put(70,85){\circle*{3}}
\put(40,85){\circle*{3}}
\put(40,85){\line(0,1){30}}
\put(40,115){\line(1,-2){30}}
\put(40,115){\line(-1,-2){30}}
\put(40,25){\line(-1,1){30}}
\put(40,25){\line(1,1){30}}
\put(40,85){\line(-1,-1){30}}
\put(40,85){\line(1,-1){30}}
\put(3,52){$x$}
\put(30,55){\circle*{3}}
\put(30,55){\line(1,3){10}}
\put(30,55){\line(1,-3){10}}
\put(50,55){\line(-1,3){10}}
\put(50,55){\line(-1,-3){10}}
\put(50,55){\circle*{3}}
\put(25,85){\circle*{3}}
\put(55,85){\circle*{3}}
\put(-4,80){$l_{x,1}$}
\put(13,80){$r_{x,1}$}
\put(16,51){$l_{0,z}$}
\put(34,52){$r_{0,z}$}
\put(72,53){$y$}
\put(73,82){$r_{y,1}$}
\put(57,80){$l_{y,1}$}
\put(43,84){$z$}
\put(38,118){$1$}
\put(40,115){\line(-1,-1){30}}
\put(40,115){\line(1,-1){30}}
\put(13,3){$D({\cal L}_2^2\oplus {\cal L}_2)$}
\end{picture}\end{tabular}\end{center}\vspace*{-15pt}
\end{example}

\begin{remark} By Remark \ref{hsuminterval}:\begin{itemize}
\item $|{\rm Filt}(D(L))|=|{\rm Filt}(L)|+(|{\rm Filt}({\cal L}_2^2)|-2)\cdot |\{(a,b)\ |\ a,b\in L,a<b,a\nprec b\}|=|{\rm Filt}(L)|+2\cdot |\{(a,b)\ |\ a,b\in L,a<b,a\nprec b\}|$;
\item $|{\rm Id}(D(L))|=|{\rm Id}(L)|+(|{\rm Id}({\cal L}_2^2)|-2)\cdot |\{(a,b)\ |\ a,b\in L,a<b,a\nprec b\}|=|{\rm Id}(L)|+2\cdot |\{(a,b)\ |\ a,b\in L,a<b,a\nprec b\}|$.\end{itemize}\label{filtiddl}\end{remark}

\begin{theorem} The lattice $D(L)$ is simple.\label{cgdl}\end{theorem}

\begin{proof} If $|L|=2$, then $D(L)=L$, so ${\rm Con}(D(L))={\rm Con}(L)=\{\Delta _L,\nabla _L\}$, and, of course, $\Delta _L\neq \nabla _L$, since $|L|>1$.

Now assume that $|L|>2$, and let $\theta \in {\rm Con}(D(L))$ such that $\theta \neq \Delta _{D(L)}$, so that, for some $x\in L$, $|x/\theta |\geq 2$, thus there exist $u,v\in x/\theta $ with $u\neq v$. Denote $y=u\wedge v\in x/\theta $ and $z=u\vee v\in x/\theta $, so that $y<z$ and $(y,z)\in \theta $. Let us analyse the following cases, of $y$ and $z$ belonging to $L$ or to $D(L)\setminus L=\{l_{a,b},r_{a,b}\ |\ a,b\in L,a<b,a\nprec b\}$.

{\bf Case 1:} $y,z\in L$. If $y=0$ and $z=1$, then $(0,1)\in \theta $, thus $\theta =\nabla _L$. If $y=0$ and $z\neq 1$, then $S=\{0,l_{0,1},z,r_{0,1},1\}\cong M_3$ is a sublattice of $L$, thus $\theta \cap S^2\in {\rm Con}(S)$, so $\theta \cap S^2=\nabla _S$ since $(0,z)\in \theta \cap S^2$, hence $(0,1)\in \theta \cap S^2\subseteq \theta $, therefore $\theta =\nabla _{D(L)}$. Analogously, if $y\neq 0$ and $z=1$, then $\theta =\nabla _L$. Finally, if $y\neq 0$ and $z\neq 1$, then $T=\{0,l_{0,z},y,r_{0,z},z\}\cong M_3$ and $U=\{y,l_{y,1},z,r_{y,1},1\}\cong M_3$ are sublattices of $L$, thus $\theta \cap T^2\in {\rm Con}(T)$ and $\theta \cap U^2\in {\rm Con}(U)$, so $\theta \cap T^2=\nabla _T$ and $\theta \cap U^2=\nabla _U$ since $(y,z)\in \theta \cap T^2\cap U^2$, therefore $(0,y)\in \theta \cap T^2$ and $(z,1)\in \theta \cap U^2$, hence $(0,y),(y,z),(z,1)\in \theta $, thus $(0,1)\in \theta $, therefore $\theta =\nabla _{D(L)}$.

{\bf Case 2:} $y\in L$ and $z\in D(L)\setminus L$, say, for instance, $z=l_{a,b}$ for some $a,b\in L$ with $a<b$ and $a\nprec b$, so that there exists a $c\in [a,b]_L\setminus \{a,b\}$. Then $y\leq a<z$, thus, since the subset $y/\theta =z/\theta $ of $D(L)$ is convex, it follows that $(a,z)\in \theta $. Also, $V=\{a,z=l_{a,b},c,r_{a,b},b\}\cong M_3$ is a sublattice of $D(L)$, thus $\theta \cap V^2\in {\rm Con}(V)$, so $\theta \cap V^2=\nabla _V$ since $(a,z)\in \theta \cap V^2$, hence $(a,b)\in \theta \cap V^2\subseteq \theta $, therefore $\theta =\nabla _{D(L)}$ by {\bf case 1}.

{\bf Case 3:} $y\in D(L)\setminus L$ and $z\in L$, say, for instance, $y=l_{a,b}$ for some $a,b\in L$ with $a<b$ and $a\nprec b$, so that there exists a $c\in [a,b]_L\setminus \{a,b\}$. Then $y<b\leq z$, thus, since the subset $y/\theta =z/\theta $ of $D(L)$ is convex, it follows that $(y,b)\in \theta $. Also, $W=\{a,y=l_{a,b},c,r_{a,b},b\}\cong M_3$ is a sublattice of $D(L)$, thus $\theta \cap W^2\in {\rm Con}(W)$, so $\theta \cap W^2=\nabla _W$ since $(y,b)\in \theta \cap W^2$, hence $(a,b)\in \theta \cap W^2\subseteq \theta $, therefore $\theta =\nabla _{D(L)}$ by {\bf case 1}.

{\bf Case 4:} $y,z\in D(L)\setminus L$, say $y=l_{a,b}$ and $z=l_{c,d}$ for some $a,b,c,d\in L$ with $a<b$, $c<d$, $a\nprec b$ and $c\nprec d$, so that $a<y<b\leq c<z<d$, thus $(c,z)\in \theta $ since the subset $y/\theta =z/\theta $ of $D(L)$ is convex, therefore $\theta =\nabla _{D(L)}$ by {\bf case 2}.

Hence ${\rm Con}(D(L))=\{\Delta _{D(L)},\nabla _{D(L)}\}$. Of course, $\Delta _{D(L)}\neq \nabla _{D(L)}$, since $|D(L)|\geq |L|>2$.\end{proof}

\begin{example} Let us consider the bounded lattice ${\cal N}=(\N ,\lcm ,\gcd ,|,1,0)$, which is complete and completely distributive, which can be easily shown by using the complete distributivity of the chain $(\N ,\leq )$ and the prime decompositions of the natural numbers. The distributivity of ${\cal N}$  ensures us that $|{\rm Con}({\cal N})|\geq \max \{|{\rm Filt}({\cal N})|,|{\rm Id}({\cal N})|\}$. As shown in \cite[Example $4.4$]{eucard}, $|{\rm Filt}({\cal N})|=|\N |=\aleph _0$ and $|{\rm Id}({\cal N})|=|{\cal P}(\N )|=2^{\aleph _0}>\aleph _0=|{\rm Filt}({\cal N})|$. Indeed, ${\rm Filt}({\cal N})={\rm PFilt}({\cal N})$, because $\N =[0)_{\cal N}$, $\{1\}=[1)_{\cal N}$ and, if we denote, for any $n\in \N ^*$ and any $p\in \P $, by $e_p(n)=\max \{k\in \N \ |\ p^k\, |\, n\}$ and we take an $F\in {\rm Filt}({\cal N})\setminus \{\{1\},\N \}$, then, by the well ordering of $(\N ,\leq )$, there exists $\displaystyle \bigwedge (F)=\prod _{p\in \P }p^{\min \{e_p(n)\ |\ n\in F\}}\in \N ^*$, so that there are only finitely many $p\in P$ with $\min \{e_p(n)\ |\ n\in F\}\neq 0$, say $p_1<p_2<\ldots <p_k$, for some $k\in \N $, are such that $\{p_1,p_2,\ldots ,p_k\}=\{p\in P\ |\ \min \{e_p(n)\ |\ n\in F\}\neq 0\}$, and there exist $n_1,n_2,\ldots ,n_k\in F$, not necessarily distinct, such that, for all $i\in [1,k]$, $e_{p_i}(n_i)=\min \{e_{p_i}(n)\ |\ n\in F\}$, hence $\displaystyle \bigwedge (F)=\lcm \{n_1,n_2,\ldots ,n_k\}\in F$, thus $\displaystyle \bigwedge (F)=\min (F)$, hence $F=[\min (F))_{\cal N}\in {\rm PFilt}({\cal N})$. The argument for ${\rm Filt}({\cal N})={\rm PFilt}({\cal N})$ in \cite{eucard} was shorter, but this one is more natural. Now, for any $P\subseteq \P $, $(P]_{\cal N}=\{n\in \N \ |\ (\exists \, k\in \N ^*)\, (\exists \, p_1,p_2,\ldots ,p_k\in P)\, (n\, |\, \lcm \{p_1,p_2,\ldots ,p_k\}=p_1\cdot p_2\cdot \ldots \cdot p_k)\}=\{n\in \N \ |\ (\exists \, k\in \N ^*)\, (\exists \, p_1,p_2,\ldots ,p_k\in P)\, (\exists \, e_1,e_2,\ldots ,e_k\in \N ^*)\, (n=p_1^{e_1}\cdot p_2^{e_2}\cdot \ldots \cdot p_k^{e_k})\}$, thus $(P]_{\cal N}\notin {\rm PId}({\cal N})$ if $|P|=\aleph _0$, and, for any $P,Q\subseteq \P $ with $P\neq Q$, $(P]_{\cal N}\neq (Q]_{\cal N}$, hence $|{\rm Id}({\cal N})|\geq |{\cal P}(\P )|=|{\cal P}(\N )|$, thus $|{\rm Id}({\cal N})|=|{\cal P}(\N )|=2^{\aleph _0}$. Hence $2^{\aleph _0}\leq |{\rm Con}({\cal N})|\leq 2^{\aleph _0}$, therefore $|{\rm Con}({\cal N})|=2^{\aleph _0}$.

If we denote by $H={\cal N}\boxplus {\cal L}_2^2$, then, according to Theorem \ref{cghsum}, since $0$ is meet--reducible and $1$ is join--reducible in ${\cal L}_2^2$, ${\rm Con}(H)={\rm Con}_{01}(H)\cup \{\nabla _H\}\cong ({\rm Con}_{01}({\cal N})\times {\rm Con}_{01}({\cal L}_2^2))\oplus {\cal L}_2=({\rm Con}_{01}({\cal N})\times \{\Delta _{{\cal L}_2^2}\})\oplus {\cal L}_2\cong {\rm Con}_{01}({\cal N})\oplus {\cal L}_2$. Unfortunately, $|{\rm Con}_{01}({\cal N})|=|{\rm Con}({\cal N})|=2^{\aleph _0}$. Indeed, let $\displaystyle \mu =eq(\{\{1\},\{0\}\}\cup \{\{\prod _{p\in P}p^{n_p}\ |\ (\forall \, p\in P)\, (n_p\in \N ^*)\}\ |\ P\subset \P ,|P|<\aleph _0\})$. It is immediate that $\mu $ is a congruence of ${\cal N}$. Let us prove that ${\rm Con}_{01}({\cal N})=(\mu ]_{{\rm Con}({\cal N})}$. For each $n\in \N $, let us denote by $P_n=\{p\in \P \ |\ p|n\}$, so that $P_1=\emptyset $, $P_0=\P $, for all $a,b\in \N $, $P_{\gcd \{a,b\}}=P_a\cap P_b$ and $P_{\lcm \{a,b\}}=P_a\cup P_b$, and $\mu =\{(u,v)\ |\ u,v\in \N ,P_u=P_v\}$. Let $\theta \in {\rm Con}({\cal N})$ such that $1/\theta =\{1\}$. Let $x,y\in \N ^*$, such that $(x,y)\in \theta $ and assume by absurdum that $P_x\neq P_y$, say $P_y\setminus P_x\neq \emptyset $. Since $x,y\in \N ^*$, $P_x$ and $P_y$ are finite non--empty subsets of $\P $. Let $z\in \N $ such that $P_z=P_y\setminus P_x$, thus $P_{\gcd \{x,z\}}=P_x\cap P_z=P_x\cap (P_y\setminus P_x)=\emptyset $ and $P_{\gcd \{y,z\}}=P_y\cap P_z=P_y\cap (P_y\setminus P_x)=P_y\setminus P_x\neq \emptyset $, hence $\gcd \{x,z\}=1$ and $\gcd \{y,z\}\neq 1$, but $\gcd \{y,z\}/\theta =\gcd \{x,z\}/\theta =1/\theta $, which is a contradiction to $1/\theta =\{1\}$. Therefore $P_x=P_y$, so $(x,y)\in \mu $, hence $\theta \subseteq \mu $. Since $1/\mu =\{1\}$, it follows that $\mu =\max \{\alpha \in {\rm Con}({\cal N})\ |\ 1/\alpha =\{1\}\}$. But we also have $0/\mu =\{0\}$, hence $\mu =\max ({\rm Con}_{01}({\cal N}))$ and thus $\{\alpha \in {\rm Con}({\cal N})\ |\ 1/\alpha =\{1\}\}={\rm Con}_{01}({\cal N})=(\mu ]_{{\rm Con}({\cal N})}$. For instance, given any $n\in \N ^*$ and any $P\subseteq \P $ with $|P|<\aleph _0$, $\displaystyle \theta _{n,P}=eq(\{\{1\},\{0\}\}\cup \{\{\prod _{p\in Q}p^{n_p}\ |\ (\forall \, p\in Q)\, (n_p\in \N ^*)\}\ |\ Q\subset \P ,|Q|<\aleph _0,Q\neq P\}\cup \{\{\prod _{p\in P}p^{n_p}\ |\ (\forall \, p\in Q)\, (n_p\in [1,n])\},\{\prod _{p\in Q}p^{n_p}\}\ |\ Q\subset \P ,|Q|<\aleph _0,Q\neq P,(\forall \, p\in Q)\, (n_p\in \N ^*)\})\in (\mu ]_{{\rm Con}({\cal N})}$. Therefore $|{\rm Con}(H)|=|{\rm Con}_{01}(H)|=|{\rm Con}_{01}({\cal N})|=|(\mu ]_{{\rm Con}({\cal N})}|\geq |\{\theta _{n,P}\ |\ P\subseteq \P ,|P|<\aleph _0\}|=|\{P\subseteq \P \ |\ |P|<\aleph _0\}|=2^{\aleph _0}$, hence $|{\rm Con}(H)|=2^{\aleph _0}$.

So the construction of the horizontal sum $H={\cal N}\boxplus {\cal L}_2^2$ does not cancel enough congruences of ${\cal N}$. We need the stronger construction $D({\cal N})$ introduced above.

By Theorem \ref{cgdl}, the lattice $D({\cal N})$ is simple. Clearly, the intervals of ${\cal N}$ having alt least three elements are $[n,0]_{\cal N}$ and $[n,kn]_{\cal N}$, with $n\in \N ^*$ and $k\in \N $, $k\geq 2$. Let us denote the set of these intervals by $\I _{\cal N}$. We can write $\I _{\cal N}$ in this way: $\I _{\cal N}=\{[n,kn]_{\cal N}\ |\ n\in \N ^*,k\in \N \setminus \{1\}\}$, so $\I _{\cal N}$ is in bijection to $\N ^*\times (\N \setminus \{1\})$, thus $|\I _{\cal N}|=|\N ^*|\cdot |\N \setminus \{1\}|=\aleph _0\cdot \aleph _0=\aleph _0$. By Remark \ref{filtiddl}, it follows that $|{\rm Filt}(D({\cal N}))|=|{\rm Filt}({\cal N})|+2\cdot |\I _{\cal N}|=\aleph _0+\aleph _0=\aleph _0$ and $|{\rm Id}(D({\cal N}))|=|{\rm Id}({\cal N})|+2\cdot |\I _{\cal N}|=2^{\aleph _0}+\aleph _0=2^{\aleph _0}$. Therefore $|{\rm Con}(D({\cal N}))|$, $|{\rm Filt}(D({\cal N}))|$ and $|{\rm Id}(D({\cal N}))|$ are pairwise distinct, more precisely $D({\cal N})$ is a simple bounded lattice, ${\rm Filt}(D({\cal N}))$ is countable and ${\rm Id}(D({\cal N}))$ is uncountable.\label{mylattice}\end{example}

\begin{example}[due to G\' abor Cz\' edli] We can modify Example \ref{mylattice} such that the resulting lattice can have all its congruences, excepting the smallest and the greatest, cancelled by the simple construction of its horizontal sum with the four--element Boolean algebra. Moreover, we can let this lattice have any infinite cardinality.

Let us denote by $(p_n)_{n\in \N }$ the sequence of the prime natural numbers, by $\displaystyle (P_{\cal N},\leq )=\prod _{n\in \N }(\N ,\leq )$, where the last $\leq $ is the natural order on $\N $, and by $Q_{\cal N}=(P_{\cal N}\amalg \{{\bf 1}\},\leq \cup \{(x,{\bf 1})\ |\ x\in P_{\cal N}\amalg \{{\bf 1}\}\})$. Clearly, $\displaystyle S_{\cal N}=\{(x_n)_{n\in \N }\in P_{\cal N}\ |\ |\{n\in \N \ |\ x_n\neq 0\}|<\aleph _0\}\amalg \{{\bf 1}\}$ is a bounded sublattice of $Q_{\cal N}$ and $\varphi :S_{\cal N}\rightarrow {\cal N}$, defined by $\varphi ({\bf 1})=0$ and $\displaystyle \varphi ((x_n)_{n\in \N })=\prod _{n\in \N }p_n^{x_n}$ for all $(x_n)_{n\in \N }\in P_{\cal N}$, is a lattice isomorphism between $S_{\cal N}$ and the lattice ${\cal N}$ in Example \ref{mylattice}. Now, if we replace, in the construction above, the chain $(\N ,\leq )$ by $0$--regular lattices with the DCC, then an analogous construction to the above shall produce a bounded lattice whose horizontal sum with ${\cal L}_2^2$ is simple.

Let $\kappa $ be an arbitrary infinite cardinality, $M$ be a set with $|M|=\kappa $ and $(A_i)_{i\in M}$ be a family of lattices with the DCC, hence with smallest elements, having $|A_i|\leq \kappa $  and the property that $\{\alpha \in {\rm Con}(A_i)\ |\ 0^{A_i}/\alpha =\{0^{A_i}\}\}=\Delta _{A_i}$ for all $i\in M$. Consider the lattice $\displaystyle P=\prod _{i\in M}A_i$ with smallest element ${\bf 0}=(0^{A_i})_{i\in M}$, the bounded lattice $Q=(P\amalg \{{\bf 1}\},\leq ^P\cup \{(x,{\bf 1})\ |\ x\in P\amalg \{{\bf 1}\}\})$ and the bounded sublattice $\displaystyle S=\{(x_i)_{i\in M}\in P\ |\ |\{i\in M\ |\ x_i\neq 0\}|<\aleph _0\}\amalg \{{\bf 1}\}$ of $Q$.

If we denote, for all $n\in \N $, by $S_n=\{(x_i)_{i\in M}\in P\ |\ |\{i\in M\ |\ x_i\neq 0^{A_i}\}|=n\}$, then, for all $n\in \N $, since, for all $i\in M$, $|{\cal P}_n(A_i)|\leq |{\cal P}_n(M)|=|M|=\kappa $, it follows that $|S_n|=\kappa $. Therefore, since $\displaystyle S=\bigcup _{n\in \N }S_n$ (and the $S_n$ are pairwise disjoint), we have $|S|=|\N |\cdot \kappa =\aleph _0\cdot \kappa =\kappa $.

Now, for all $T\in {\cal P}_{<\kappa }(M)$, let $S_T=\{(x_i)_{i\in M}\in P\ |\ \{i\in M\ |\ x_i\neq 0^{A_i}\}\subseteq T\}$. Clearly, if $T,U\in {\cal P}_{<\kappa }(M)$ with $T\neq U$, then $S_T\neq S_U$, thus $|\{S_T\ |\ T\in {\cal P}_{<\kappa }(M)\}|=|{\cal P}_{<\kappa }(M)|=|M|^{\kappa }=\kappa ^{\kappa }=2^{\kappa }$. It is immediate that, for all $T\in {\cal P}_{<\kappa }(M)$, $S_T\in {\rm Id}(S)$, hence $2^{\kappa }\leq |{\rm Id}(S)|\leq |{\cal P}(S)|=2^{\kappa }$, hence $|{\rm Id}(S)|=2^{\kappa }>\kappa $.

Now let us prove that all filters of $S$ are principal. Let $F\in {\rm Filt}(S)\setminus \{[{\bf 1})_S\}$, $f=(f_i)_{i\in M}\in P\cap F=F\setminus \{{\bf 1}\}$ and $N_f=\{i\in M\ |\ f_i\neq 0^{A_i}\}$, so that $|N_f|<\aleph _0$ by the definition of $S$. If ${\bf 0}\in F$, then $F=[{\bf 0})_S$. Now assume that ${\bf 0}\notin F$, so that $f\neq {\bf 0}$ and thus $N_f\neq \emptyset $. Let $\displaystyle p_{N_f}:P\rightarrow \prod _{i\in N_f}A_i$ be the canonical projection: $p_{N_f}((x_i)_{i\in M})=(x_j)_{j\in N_f}$ for all $(x_i)_{i\in M}\in P$. It is straightforward that $\displaystyle p_{N_f}\mid _{(f]_S}:(f]_S\rightarrow \prod _{i\in N_f}A_i$ is an injection, hence $(f]_S\cong p_{N_f}((f]_S)$, which is a sublattice of the finite direct product $\displaystyle \prod _{i\in N_f}A_i$, hence it has the DCC. Thus the bounded lattice $(f]_S$ has the DCC, hence the set $\{f\wedge g\ |\ g\in F\}\subseteq (f]_S$ has minimal elements; let $g^*\in F$ such that $f^*=f\wedge g^*\in (f]_S\subseteq F$ is a minimal element of this set. Since $f^*\in F\in {\rm Filt}(S)$, we have $[f^*)_S\subseteq F$. Assume by absurdum that $F\nsubseteq [f^*)_S$, so that there exists an $h\in F$ with $f^*\nleq h$, thus $f^*\neq f^*\wedge h$ and hence $f^*>f^*\wedge h=f\wedge g^*\wedge h$, which contradicts the minimality of $f^*$ since $g^*\wedge h\in F$. Therefore $F\subseteq [f^*)_S$, hence $F=[f^*)_S\in {\rm PFilt}(S)$, thus ${\rm Filt}(S)={\rm PFilt}(S)$ and hence $|{\rm Filt}(S)|=|S|=\kappa $.

For any $\theta \in {\rm Con}(S)$ and any $i\in M$, denote by $pr_i(\theta )=\{(a,b)\in A_i^2\ |\ (\exists \, ((x_i)_{i\in M},(y_i)_{i\in M})\in \theta \cap P^2)\, (x_i=a,y_i=b)\}\in {\rm Con}(A_i)$. Now let $\theta \in {\rm Con}_{01}(S)$. Then, for all $i\in M$, $0^{A_i}/pr_i(\theta )=\{0^{A_i}\}$, so that $pr_i(\theta )=\Delta _{A_i}$. Since $P\cap S=S\setminus \{{\bf 1}\}$ is a sublattice of $S$, $\theta \cap (S\setminus \{{\bf 1}\})^2\in {\rm Con}(S\setminus \{{\bf 1}\})$, and, clearly, $\displaystyle \theta \cap (S\setminus \{{\bf 1}\})^2\subseteq \prod _{i\in M}pr_i(\theta )=\prod _{i\in M}\Delta _{A_i}=\Delta _P$, therefore $\displaystyle \theta \cap (S\setminus \{{\bf 1}\})^2=\Delta _{S\setminus \{{\bf 1}\}}$. But ${\bf 1}/\theta =\{{\bf 1}\}$, hence $x/\theta =\{x\}$ for all $x\in (S\setminus \{{\bf 1}\})\cup \{{\bf 1}\}=S$, that is $\theta =\Delta _S$. Therefore ${\rm Con}_{01}(S)=\{\Delta _S\}$.

If we denote by $H=S\boxplus {\cal L}_2^2$, then, since $0$ is meet--reducible and $1$ is join--reducible in ${\cal L}_2^2$, by Theorem \ref{cghsum} it follows that ${\rm Con}(H)={\rm Con}_{01}(H)\cup \{\nabla _H\}\cong ({\rm Con}_{01}(S)\times {\rm Con}_{01}({\cal L}_2^2))\oplus {\cal L}_2=(\{\Delta _S\}\times \{\Delta _{{\cal L}_2^2}\})\oplus {\cal L}_2\cong {\cal L}_1\oplus {\cal L}_2\cong {\cal L}_2$, so the bounded lattice $H$ is simple. By Remark \ref{filtidhsum}, $|{\rm Filt}(H)|=|{\rm Filt}(S)|=\kappa $ and $|{\rm Id}(H)|=|{\rm Id}(S)|=2^{\kappa }$. Hence $|{\rm Con}(H)|$, $|{\rm Filt}(H)|$ and $|{\rm Id}(H)|$ are pairwise distinct.\label{gcslattice}\end{example}

\section*{Acknowledgements}

This work was supported by the research grant {\em Propriet\`a d`Ordine Nella Semantica Algebrica delle Logiche Non--classiche} of Universit\`a degli Studi di Cagliari, Regione Autonoma della Sardegna, L. R. $7/2007$, n. $7$, $2015$, CUP: ${\rm F}72{\rm F}16002920002$.

I thank Antonio Ledda and Francesco Paoli for useful discussions about the issues addressed in this paper. I also thank G\' abor Cz\' edli for pointing out the error in a previous version of the proof from Example \ref{mylattice}, as well as constructing Example \ref{gcslattice}.


\begin{thebibliography}{99}
 
\bibitem{agl} P. Agliano, Prime Spectra in Modular Varieties, {\em Algebra Universalis} 30 (1993), 581--597.

\bibitem{blyth} T. S. Blyth, {\em Lattices and Ordered Algebraic Structures}, Springer--Verlag, Universitext, London, 2005.

\bibitem{bruhar} G. Bruns, J. Harding, Algebraic Aspects of Orthomodular Lattices, in: B. Coecke, D. Moore, A. Wilce (eds), {\rm Current Research in Operational Quantum Logic. Fundamental Theories of Physics}, vol. 111, Springer, Dordrecht, 2000.

\bibitem{eucard} C. Mure\c san, On the Cardinalities of the Sets of Congruences, Ideals and Filters of a Lattice, {\em Analele Universit\u a\c tii Bucure\c sti. Seria Informatic\u a. Proceedings of the Workshop Days of Computer Science (DACS) 2015 LXII, affiliated workshop of the 11th edition of the conference Computability in Europe} (2015), 55--68, University of Bucharest, Bucharest, Romania.\end{thebibliography}
\end{document}